\documentclass[12pt]{amsart}
\usepackage{amsmath,amsthm,amssymb,amsfonts,enumerate,color}
\usepackage{amsmath, amsthm, amsfonts, amssymb, color}
\usepackage{mathrsfs}
\usepackage{ amsmath, amsthm, amsfonts, amssymb, color}
 \usepackage{mathrsfs}
\usepackage{amsfonts, amsmath}
 \usepackage{amsmath,amstext,amsthm,amssymb,amsxtra}
 \usepackage{txfonts} 
 \usepackage[colorlinks, citecolor=blue,pagebackref,hypertexnames=false]{hyperref}
 \allowdisplaybreaks
 \usepackage{pgf,tikz}
\begin{document}
\newcommand{\norm}[1]{\left\Vert#1\right\Vert}
\newcommand{\abs}[1]{\left\vert#1\right\vert}
\newcommand{\set}[1]{\left\{#1\right\}}
\newcommand{\Real}{\mathbb{R}}
\newcommand{\supp}{\operatorname{supp}}
\newcommand{\card}{\operatorname{card}}
\renewcommand{\L}{\mathcal{L}}
\renewcommand{\P}{\mathcal{P}}
\newcommand{\T}{\mathcal{T}}
\newcommand{\A}{\mathbb{A}}
\newcommand{\K}{\mathcal{K}}
\renewcommand{\S}{\mathcal{S}}
\newcommand{\blue}[1]{\textcolor{blue}{#1}}
\newcommand{\red}[1]{\textcolor{red}{#1}}
\newcommand{\Id}{\operatorname{I}}
\def\SL{\sqrt[m] L}
\def\RR{\mathbb R}
\def \R   {\mathbb{R}^n}
\def \RN   {\mathbb{R}^n}
\newcommand\C{{\mathbb C}}
\newcommand\CC{\mathbb{C}}
\newcommand\NN{\mathbb{N}}

\newtheorem{thm}{Theorem}[section]
\newtheorem{prop}[thm]{Proposition}
\newtheorem{cor}[thm]{Corollary}
\newtheorem{lem}[thm]{Lemma}
\newtheorem{lemma}[thm]{Lemma}
\newtheorem{exams}[thm]{Examples}
\theoremstyle{definition}
\newtheorem{defn}[thm]{Definition}
\newtheorem{rem}[thm]{Remark}

\numberwithin{equation}{section}

\title[Convergence of spectral sums for self-adjoint operators ]
{Almost everywhere convergence of spectral sums for self-adjoint operators}

 \author[Peng Chen, Xuan Thinh Duong and Lixin Yan]{Peng Chen, Xuan Thinh Duong and Lixin Yan}

\address{Peng Chen, Department of Mathematics, Sun Yat-sen
 University, Guangzhou, 510275, P.R. China}
 \email{chenpeng3@mail.sysu.edu.cn}
 \address {Xuan Thinh Duong, Department of Mathematics, Macquarie University, NSW 2109, Australia}
\email{xuan.duong@mq.edu.au}
 \address{Lixin Yan, Department of Mathematics, Sun Yat-sen   University,
Guangzhou, 510275, P.R. China}
\email{mcsylx@mail.sysu.edu.cn}

 \title[Convergence of spectral sums for self-adjoint operators ]
{Almost everywhere convergence of spectral sums for self-adjoint operators}

\date{\today}
\subjclass[2000]{42B15, 42B25,   47F05.}
\keywords{Almost everywhere convergence,  the spherical partial sums,
 non-negative self-adjoint operators, Rademacher-Menshov theorem,  Plancherel-type estimate.}

\begin{abstract}
Let $L$
be a non-negative self-adjoint operator acting on the
space $L^2(X)$, where $X$ is a metric measure space.
Let  ${  L}=\int_0^{\infty} \lambda dE_{  L}({\lambda})$  be the spectral resolution of  ${  L}$
   and
    $
   S_R({ L})f=\int_0^R dE_{  L}(\lambda) f
   $
 denote  the spherical partial sums  in terms of the resolution of ${  L}$.
In this article we give a sufficient condition on $L$
such that
$$\lim_{R\rightarrow \infty}  S_R({ L})f(x)  =f(x),\ \  {\rm a.e.}
$$
for any $f$ such that ${\rm log } (2+L) f\in L^2(X)$.

 These results are applicable to large classes of operators including
Dirichlet operators on smooth bounded domains, the Hermite operator  and Schr\"odinger operators with inverse square potentials.
 \end{abstract}

\maketitle


\section{Introduction}
 \setcounter{equation}{0}

 The almost-everywhere convergence of the spherical partial sums
 \begin{eqnarray*}
 S_Rf(x)=\int_{|\xi|\leq R}  \widehat{f}(\xi)   e^{2\pi i x\cdot \xi} d\xi
\end{eqnarray*}
on $L^2(\mathbb R^n)$ is a well-known classical problem in  Fourier analysis, where  $\widehat{f}$ denotes the Fourier transform of
$f.$
In the one-dimensional case $n=1$,  a celebrated theorem of Carleson\cite{Car} states  that for $f\in L^2(\mathbb R)$,
 \begin{eqnarray*}
 \lim_{R\to \infty}S_Rf(x)= f(x)\ \ \ \ \ {\rm for \ almost \ every \ } x.
\end{eqnarray*}
For $n\geq 2$,  Carbery and Soria \cite[Theorem 3]{CS} proved that $\lim\limits_{R\rightarrow \infty}S_Rf(x)=f(x) $
almost everywhere for any $f$ such that  ${\rm log}(2+\Delta) f\in L^2(\mathbb R^n)$, where $\Delta=-\sum_{i=1}^n\partial_{x_i}^2$ denotes the classical
Laplace operator on ${\mathbb R}^n.$

 In \cite{MMP},  Meaney, M\"uller and Prestini  extended the result of Carbery and Soria
   to  arbitrary right-invariant sub-Laplacian ${\mathbb L}$ on a connected Lie group ${\mathbb G}$.
   Let  ${\mathbb L}=\int_0^{\infty} \lambda \ dE_{\mathbb L}({\lambda})$  be the spectral resolution of  ${\mathbb L}$
   and
   $$
   S_R({\mathbb L})f(x)=\int_0^R dE_{\mathbb L}(\lambda) f(x)
   $$
 denote  the spherical partial sums  in terms of the resolution of ${\mathbb L}$.
They showed that $S_R({\mathbb L})f(x) $ converges a.e. to $f(x)$ as $R\rightarrow \infty $
when ${\rm log}(2+{\mathbb L}) f\in L^2({\mathbb G})$. Their proof is based on the Rademacher-Menshov theorem (\cite{CMP, Z}).
It also employs an extension of a Plancherel theorem as in \cite{HJ} and \cite{Ch}  to arbitrary connected Lie groups
${\mathbb G}$, which says  that for any Borel measurable essentially bounded function $F$ on $[0, \infty)$ and for  the spectral
multiplier
  $F({\mathbb L})f =K_F\ast f$ corresponds a unique distribution   $K_F$,
  there exists a unique $\sigma$-finite positive Borel measure $\omega$ on $[0, \infty)$ such that the following holds:
 \begin{eqnarray}\label{e1.1}
\|K_F\|_2^2=\int_0^{\infty} | F(\lambda)|^2 d\omega(\lambda).
\end{eqnarray}

 \smallskip

In this article we   assume that  $(X,d,\mu)$ is  a separable metric measure
space, that is  $\mu$ is a Borel measure with respect to the
topology defined by the metric $d$. Next let $B(x, r) = \{y\in X: d(x, y) < r\}$ be the open
ball with center $x\in X$ and radius $r>0$.   Given a  subset $E\subseteq X$, we  denote by  $\chi_E$   the characteristic
function of   $E$ and  set
$
P_Ef(x)=\chi_E(x) f(x).
$
We consider   a non-negative self-adjoint operator $L$ acting on $L^2(X)$.
  Such an operator admits a spectral resolution $E_L(\lambda)$ and
we define the spherical partial sums for $L$ by
$$
S_R(L)f(x) = \int_0^R dE_L(\lambda)f(x).
$$
The  aim  of this article is to  investigate
 when it is possible to replace condition
\eqref{e1.1} in the Meaney-M\"uller-Prestini theorem  by other suitable condition  to study
  almost everywhere convergence of spherical partial sums  in the
general setting of abstract operators rather than in a specific setting of group
invariant operators acting on Lie groups.
To do it,  we recall that in  (\cite[3.1]{DOS}), Duong,  Ouhabaz  and
  Sikora introduced the so-called Plancherel-type estimate   to establish the sharp H\"ormande-type spectral multiplier theorems  for
 $L$. We say  that $L$ satisfies the Plancherel-type estimate if there exists $C>0$ such that
  for all $M>0$, $y\in X$ and all Borel functions $F$  such that\, {\rm supp} $F\subseteq [0, M]$,
  \begin{eqnarray}\label{e1.2}
\int_X |K_{F(\sqrt[m]{L})}(x,y)|^2 d\mu(x) \leq {\frac{C}{\mu(B(y, M^{-1}))}} \|F(M\cdot)\|^2_{L^2},
\end{eqnarray}
where $K_{F(\sqrt[m]{L})}(x,y): X \times X\to {\mathbb C}$ denotes the kernel of the integral operator $F(\sqrt[m]{L}),$
and $m$ is positive constant and $m\geq 2.$ For the standard Laplace operator $\Delta $ on ${\mathbb R^n}$,
 it is well-known (\cite[Proposition 2.4]{COSY}) that condition \eqref{e1.2} is equivalent to
the $(1,2)$ restriction estimate  of Stein-Tomas, i.e.
$$
 \|dE_{\sqrt{\Delta}}(\lambda)\|_{L^1\to L^\infty} \leq C\lambda^{n-1}.
$$
Alternative form of the Plancherel-type estimate was introduced in \cite[(4.3)]{KU} by Kunstmann and
Uhl,  and can be formulated in the following way:
\begin{eqnarray}\label{e1.3}
\|F(\SL)P_{B(x, 1/M)} \|_{2\to 2} \leq   \|F(M\cdot)\|_{2}
\end{eqnarray}
for all $M>0, x\in X$ and all bounded Borel functions $F$ with $\supp F\subset [0, M]$.
Note that
\begin{eqnarray*}
\|F(\SL)P_{B(x, 1/M)} \|_{2\to 2} \leq  \|F(\SL)P_{B(x, 1/M)} \|_{1\to 2}  \| P_{B(x, 1/M)} \|_{2\to 1},
\end{eqnarray*}
so, by H\"older's inequality,  estimate \eqref{e1.2}   implies \eqref{e1.3} provided that
$X$ is a space of homogeneous type (see Section 2 below). For more information
about  \eqref{e1.2} and \eqref{e1.3}, we refer to \cite{COSY, DOS, KU} and the references therein.

Motivated by the Plancherel-type estimates \eqref{e1.2} and \eqref{e1.3} above, we have the following result.

\begin{thm}\label{th1.1}
Let $(X,d, \mu)$ be a metric measure space and $L$ satisfies
the   Plancherel-type estimate: for all compact subset $K$, there exist positive constants
$C_K$ and $a$ such that for all $M>1$, all Borel functions $F$
 with $\supp F\subset [M/4, M]$
\begin{align}\label{e1.4}
\|F(L)\chi_{K}\|_{2\to 2}\leq C_K M^{a}\|F(M\cdot)\|_{L^2}.
\end{align}
If \,$\log(2+L)f\in L^2(X)$, then
$$
\lim_{R\to \infty} S_R(L) f(x)=f(x)
$$
for almost every $x\in X$. Moreover, for every compact subset $K$ of $X$ there exists a constant $C_K>0$ such that
\begin{align}\label{e1.5}
\int_K \big|\sup_{R>0}|S_R(L) f(x)|\big|^2d\mu(x)\leq C_K \|\log(2+L)f\|_2^2.
\end{align}
\end{thm}

\medskip

It is not difficult to see that  the   Plancherel type estimate \eqref{e1.4} implies that
  the set of point spectrum of $L$ is empty in $(1/4, \infty)$.   Indeed, one has, for   $0\leq \lambda< M$, 
$
\|1\!\!1_{\{\lambda \}  }( {L}\,) \chi_{K}\|_{2\to 2}
\leq C_K  \|1\!\!1_{\{\lambda \} }(M\cdot)\|_{2} =0,
$
and thus $1\!\!1_{\{\lambda\} }( {L}\,)=0$. Since  $\sigma(L)\subseteq [0, \infty)$, it is clear that the point spectrum of $L$
is empty in $(1/4, \infty)$.
In particular, \eqref{e1.4} does  not hold  for elliptic
operators on compact manifolds or for the harmonic oscillator.
 In order to treat these cases,  we will   prove the following result.

\begin{thm}\label{th1.2}
Let $(X,d,\mu)$ be a metric measure space and assume that  the spectrum of $L$ is purely discrete, i.e. the essential spectrum is empty. Let
$\lambda_1<\lambda_2<\cdots\lambda_k<\cdots$ be all the different eigenvalues of $L$.
Assume that there exist constants $A,a>0$ such that for large enough natural number $k$
\begin{align}\label{e1.6}
k\leq A\lambda_k^a.
\end{align}
If \,$\log(2+L)f\in L^2(X)$, then
$$
\lim_{R\to \infty} S_R(L) f(x)=f(x)
$$
for almost every $x\in X$. Moreover, for every compact subset $K$ of $X$ there exists a constant $C_K>0$ such that
\begin{align}\label{e1.7}
\int_K \big|\sup_{R>0}|S_R(L) f(x)|\big|^2d\mu(x)\leq C_K \|\log(2+L)f\|_2^2.
\end{align}
\end{thm}

\smallskip

  We would like to mention that in Theorem~\ref{th1.1}, when $X$ is a space   of homogeneous type,  either
  \eqref{e1.2} or \eqref{e1.3} implies estimate \eqref{e1.4}, see
  Lemma~\ref{le2.2} below.
  There are several examples of operators discussed in   \cite{COSY, DOS, KU}  which  satisfy
 the Plancherel-type estimate  \eqref{e1.2} or \eqref{e1.3}.
  In particular, \eqref{e1.2} holds for positive definite self-adjoint right invariant operators
  and quasi-homogeneous operators acting on a homogeneous group, see \cite[Section  7.1]{DOS}.
  However, it is not clear for us whether or not estimate \eqref{e1.2} holds  for the right-invariant sub-Laplacian ${\mathbb L}$
  on a connected Lie group ${\mathbb G}$.

Note that in Theorem~\ref{th1.2}, if   the number $N(\lambda)$ of eigenvalues in $[0,\lambda]$,
counted with the multiplicities of each eigenvaule, satisfies
\begin{eqnarray}\label{e1.8}
N(\lambda)\leq A\lambda^{a},
\end{eqnarray}
 then for eigenvalue $\lambda_k$,
$$
k\leq N(\lambda_k)\leq A\lambda_k^{a}.
$$
Estimate \eqref{e1.8} can be derived from the Weyl formula for $L$, see for examples,  Sections 5.1 and 5.2
below. As pointed in \cite[p. 470]{DOS}, in the case of group invariant operators on
compact Lie groups the Plancherel-type estimates and the sharp Weyl formula
are equivalent.

Our  Theorems ~\ref{th1.1} and  ~\ref{th1.2} are applicable to large classes of operators including
Dirichlet operators on bounded domains, the Hermite operator   and Schr\"odinger operators with the inverse square potentials.
See Section 5 below for details.

\section{Preliminary  results}
 \setcounter{equation}{0}

As mentioned in Introduction, the proofs of Theorems  \ref{th1.1} and \ref{th1.2} are based
 on the following Rademacher-Menshov Theorem (see \cite{A,Z}).

\begin{thm}[Rademacher-Menshov Theorem]\label{th2.1}
Suppose that $(X,\mu)$ is a positive measure space. There is a positive constant $c$ with the following property:
For each orthogonal subset $\{f_k:k\in \mathbb{N}\}$ in $L^2(X)$ satisfying
\begin{align}\label{e2.1}
\sum_{k=0}^\infty (\log(2+k))^2 \|f_k\|_2^2<\infty,
\end{align}
the maximal function
$$
F^*(x):=\sup_{N\in \NN}|\sum_{k=0}^N f_k(x)|
$$
is in $L^2(X)$, and
\begin{align}\label{e2.2}
\|F^*\|^2_2\leq c\sum_{k=0}^\infty (\log(2+k))^2 \|f_k\|_2^2.
\end{align}

In particular, when \eqref{e2.1} holds, then the series $\sum_{n=1}^{\infty}
f_k(x)$ converges almost everywhere on $X.$
\end{thm}

\begin{proof}
For the proof , we refer to Theorem XIII.10.21 from \cite{Z}, Proposition 2.3.1, and Theorem 2.3.2 from \cite[pp. 79-80]{A}.
\end{proof}

Following   \cite[Chapter 3]{CW}), a space of homogeneous type $(X, d, \mu)$ is a set $X$ together
with a  metric $d$ and a nonnegative measure $\mu$ on $X$ such that $\mu(B(x, r)) < \infty$
for all $x\in X$ and all $r > 0$, and  there  exists a constant $C>0$ such that
\begin{align}\label{e2.3}
V(x,2r)\leq C V(x, r)\quad \forall\,r>0,\,x\in X,
\end{align}
where $ V(x, r)=\mu(B(x,r))$.
If this is the case, there exist  $C, n$ such that for all $\lambda\geq 1$ and $x\in X$
\begin{align}\label{e2.4}
V(x, \lambda r)\leq C\lambda^n V(x,r).
\end{align}
for some $c, n>0$ uniformly for all $\lambda\geq 1$ and $x\in{
X}$. The parameter
$n$ is a measure of the dimension of the space. There also exist $c$ and
$N, 0\leq N\leq n$ so that
\begin{equation}
V(y, r)\leq c\bigg( 1+\frac{d(x,y)}{r}\bigg )^NV(x,r)
\label{e2.5}
\end{equation}
uniformly for all $x,y\in {  X}$ and $r>0$. Indeed, the
property (\ref{e2.5}) with
$N=n$ is a direct consequence of triangle inequality of the metric
$d$ and the strong homogeneity property. In the cases of Euclidean spaces
${\mathbb R}^n$ and Lie groups of polynomial growth, $N$ can be
chosen to be $0$.

As mentioned in Introduction, estimate \eqref{e1.2} implies  \eqref{e1.3} when $X$ is a space   of homogeneous type.
  Now we discuss the relationship between two    Plancherel-type estimates  \eqref{e1.3} and \eqref{e1.4}.
  We have the following result.

\begin{lemma}\label{le2.2}
Let $X$ be a space   of homogeneous type.
Suppose that the operator $L$ satisfies the condition  \eqref{e1.3}, then estimate \eqref{e1.4} holds.
\end{lemma}

\begin{proof}  
Let $M>1$ and let $F$ be a  Borel functions   such that\, {\rm supp} $F\subseteq [M/4, M]$.
Since $K$ is compact, we have a ball $B=B(x_K,r_K)$ with $r_K\geq 1/M$ such
 that $K\subset B(x_K,r_K)$. Then we take a function $G(\lambda)=F(\lambda^m)$ such that $G(\sqrt[m]L)=F(L)$ and
so $\supp G\subset [0,M^{1/m}]$.

For every $1/M>0$, we choose  a sequence $(x_i)_{i=1}^{N_M} \in B(x_K,r_K)$ for some $N_M<\infty$ such that
$d(x_i,x_j)> {1/2M}$ for $i\neq j$ and $\sup_{x\in X}\inf_i d(x,x_i)
\le {1/2M}$. Such a sequence exists because $X$ is separable.
Set  $ B(x_K,r_K)\subseteq \bigcup_{i\in\in N_M} B(x_i, 1/M)$.
 Note that for every $1\leq i, j\leq  N_M$
 $$
 V(x_i, 1/M)\leq C\bigg( 1+\frac{d(x_i, x_j)}{1/M}\bigg )^NV(x_j,1/M)\leq C(r_KM)^n V(x_j,1/M).
 $$
Without loss of generity,  we assume that $x_1=x_K.$ Then we have
 \begin{eqnarray*}
 N_M (r_KM)^{-n} V(x_1, 1/M)&\leq& C\sum_{i\in N_M} V(x_i, 1/M)
  \leq  C V(x_K, r_K)\\
  &\leq& C {V(x_K, r_K)\over V(x_K, 1/M)} V(x_K, 1/M) \\
 &\leq&     C  (r_KM)^{n} V(x_1, 1/M),
 \end{eqnarray*}
so  $  N_M  \leq Cr_K^{2n}M^{2n}$.
 Therefore,
\begin{align*}
\|F(L)\chi_{K}\|_{2\to 2}&\leq \|G(\sqrt[m]L)\chi_{B(x_K,r_K)}\|_{2\to 2}\\
&\leq \sum_{i\in N_M}\|G(\sqrt[m]L)\chi_{B(x_i,1/R)}\|_{2\to 2}\\
&\leq C\sum_{i} \|G(M^{1/m}\cdot)\|_{L^2}\\
&\leq Cr_K^{2n}M^{2n}\|F(M\lambda^m)\|_{L^2}\\
&\leq Cr_K^{2n}M^{2n}\|F(M\cdot)\|_{L^2}.
\end{align*}
This completes the proof of Lemma~\ref{le2.2}.
\end{proof}

 \smallskip
 \begin{rem}
Note that in our Theorems~\ref{th1.1} and \ref{th1.2}, we assume that  $(X,d,\mu)$ is a separable metric measure space,
and we do not need the assumption that
 $X$ is a space   of homogeneous type.
\end{rem}

\medskip

\section{Proof of Theorem~\ref{th1.2} }
 \setcounter{equation}{0}

 To show Theorem~\ref{th1.2}, we note that
  the spectrum of $L$ is purely discrete and the eigenvalues satisfy condition~\eqref{e1.6}. In this case,
$$
S_R(L) f(x)=\sum_{k=0}^{[R]} \sum_{i:\lambda_k}\langle f,\phi_{k,i}\rangle\phi_{k,i}(x)=\sum_{k=0}^{[R]} P_k f(x)
$$
where $\{\phi_{k,i}(x)\}$ are the eigenfunctions corresponding to the eigenvalue $\lambda_k$ and $[R]$ denotes the largest integer number such that $\lambda_{[R]}\leq R$.
From condition~\eqref{e1.6}, we see  that there exists constant $C>0$ such that
$$
\log(2+k)\leq C\log(2+\lambda_k).
$$
Taking $f_k =P_k(L) f $ in \eqref{e2.1}, we have
\begin{align*}
\sum_{k=0}^\infty (\log(2+k))^2 \|P_k(L) f\|_2^2&=  \sum_{k=0}^\infty (\log(2+k))^2 \sum_{i:\lambda_k}\langle f,\phi_{k,i}\rangle^2\\
&\leq C  \sum_{k=0}^\infty (\log(2+\lambda_k))^2 \sum_{i:\lambda_k}\langle f,\phi_{k,i}\rangle^2\\
&= C  \sum_{k=0}^\infty  \sum_{i:\lambda_k}\langle f,\log(2+\lambda_k)\phi_{k,i}\rangle^2\\
&= C  \sum_{k=0}^\infty  \sum_{i:\lambda_k}\langle f,\log(2+L)\phi_{k,i}\rangle^2\\
&= C  \sum_{k=0}^\infty  \sum_{i:\lambda_k}\langle \log(2+L)f,\phi_{k,i}\rangle^2\\
&= C \|\log(2+L)f\|_2^2.
\end{align*}
By the Rademacher-Menshov Theorem~\ref{th2.1},
\begin{eqnarray*}
\big\|\sup_{R>0}|S_R(L) f(x)|\big\|_2^2=\big\|\sup_{N\in \NN}|\sum_{k=0}^N P_k(L) f|\big\|_2^2&\leq& C\sum_{k=0}^\infty (\log(2+k))^2 \|P_k(L) f\|_2^2\\
&\leq& C \|\log(2+L)f\|_2^2,
\end{eqnarray*}
which completes the proof of \eqref{e1.7} in Theorem~\ref{th1.2}.
  \hfill{}$\Box$

\medskip

\section{Proof of Theorem~\ref{th1.1}}\label{continuous}
 \setcounter{equation}{0}
 To show \eqref{e1.5} in Theorem~\ref{th1.1}, we write $\lambda_k=k^{1/(2a)}$ and
$$
P_k f(x):=E_L(\lambda_{k-1},\lambda_{k}]f(x)=S_{\lambda_k}(L)f(x)-S_{\lambda_{k-1}}(L)f(x),
$$
where $a$ is the constant in condition~\eqref{e1.4}.
Then
$$
\sup_{R>0}|S_R f(x)|\leq \sup_{N\in \NN}|\sum_{k=0}^N P_k f(x)|+\sup_{k\in \NN}
\left(\sup_{\lambda_{k}\leq r<\lambda_{k+1}}|E_L(\lambda_{k},r]f(x)|\right)=:{\rm I}+{\rm II}.
$$

For the first term ${\rm I}$, we note that
there exists a  constant $C=C(a)>0$  such that
$$
\log(2+k)\leq C\log(2+\lambda_k).
$$
Following an argument as in Theorem~\ref{th1.2}, we take
  $f_k =P_k f $ in  \eqref{e2.1} to get
\begin{align*}
\sum_{k=1}^\infty (\log(2+k))^2 \|P_k f\|_2^2&=  \sum_{k=1}^\infty (\log(2+k))^2 \int_{(\lambda_{k-1},\lambda_k]}d\langle E_L(\lambda)f,f\rangle\\
&\leq 4\sum_{k=1}^\infty (\log(2+k-1))^2 \int_{(\lambda_{k-1},\lambda_k]}d\langle E_L(\lambda)f,f\rangle\\
&\leq C  \sum_{k=1}^\infty \int_{(\lambda_{k-1},\lambda_k]}(\log(2+\lambda_{k-1}))^2 d\langle E_L(\lambda)f,f\rangle\\
&\leq C  \sum_{k=1}^\infty \int_{(\lambda_{k-1},\lambda_k]}(\log(2+\lambda))^2 d\langle E_L(\lambda)f,f\rangle\\
&\leq C \|\log(2+L)f\|_2^2.
\end{align*}
For $k=0$, it is clear that $(\log(2+0))^2 \|f_0\|_2^2\leq C\|f\|^2_2\leq C \|\log(2+L)f\|_2^2$.
By the Rademacher-Menshov Theorem~\ref{th2.1},
\begin{align}\label{e4.1}
\left\|\sup_{N\in \NN}|\sum_{k=0}^N P_k f(x)|\right\|_2^2\leq C\sum_{k=0}^\infty (\log(2+k))^2 \|P_k f\|_2^2
\leq C\|\log(2+L)f\|_2^2.
\end{align}

Let us estimate the   term ${\rm II}$. To do it, it follows from the fact that $\ell^2\subseteq \ell^\infty$ and
the dual space of $L^2(K,L^1[\lambda_{k},\lambda_{k+1}])$ is $L^2(K,L^\infty[\lambda_{k},\lambda_{k+1}])$ (see \cite{BP}) that
\begin{align}\label{e4.2}
\left\|\sup_{k\in \NN} \left(\sup_{\lambda_{k}\leq r<\lambda_{k+1}}|E_L(\lambda_{k},r]f|\right)\right\|^2_{L^2(K)}
&\leq C \sum_{k\in \NN}\left\| \sup_{\lambda_{k}\leq r<\lambda_{k+1}}|E_L(\lambda_{k},r]f|\right\|^2_{L^2(K)}\nonumber\\
& \leq C\sum_{k\in \NN}\sup_{\|g\|_{L^2(K,L^1[\lambda_{k},\lambda_{k+1}])}=1}
\left|\int_K\int_{[\lambda_{k},\lambda_{k+1}]} E_L(\lambda_{k},r]f(x)g(r,x)drdx\right|^2.
\end{align}
Integration by part gives us
\begin{align} \label{e4.3}
&\int_{[\lambda_{k},\lambda_{k+1}]} E_L(\lambda_{k},r]f(x)g(r,x)dr\nonumber\\
&= E_L(\lambda_{k},\lambda_{k+1}]f(x)\int_{\lambda_{k}}^{\lambda_{k+1}}g(s,x)ds-\int_{\lambda_{k}}^{\lambda_{k+1}}
\left(\int_{\lambda_{k}}^rg(s,x)ds\right) d E_L(\lambda_{k},r]f(x),
\end{align}
where the equality \eqref{e4.3} makes sense in   $L^2(X).$
To go on,
we make a partition of the interval $(\lambda_k,\lambda_{k+1}]$: $\lambda_k=\lambda_{k,0}<\lambda_{k,1}<\ldots<\lambda_{k,J}=\lambda_{k+1}$.
From   the Plancherel type estimate \eqref{e1.2}, we see that
\begin{align*}
\|E_L(\lambda_{k,j-1},\lambda_{k,j}]f\|_{L^2(K)}&\leq
\|\chi_K E_L(\lambda_{k,j-1},\lambda_{k,j}]\|_{2\to 2} \|E_L(\lambda_{k,j-1}, \lambda_{k,j}]f\|_2\nonumber\\
&\leq C_K \lambda_{k,j}^{a} \Big\|\chi_{\left(\frac{ \lambda_{k,j-1}}{\lambda_{k,j}}, 1\right]}\Big\|_2 \|E_L(\lambda_{k,j-1},\lambda_{k,j}]f\|_2\nonumber\\
&\leq C_K \lambda_{k,j}^{a-\frac{1}{2}} (\lambda_{k,j}-\lambda_{k,j-1})^{\frac{1}{2}}  \|E_L(\lambda_{k,j-1},\lambda_{k,j}]f\|_2,
\end{align*}
where   $C_{K}$ is a constant depending on $K$ only, but does not  depend on $j$ and $k$.
This, in combination with the properties
  of Riemann-Stieltjes integral and  the Fatou Lemma, yields that
  \begin{eqnarray}\label{e4.4}
&&\left|\int_K \int_{\lambda_{k}}^{\lambda_{k+1}}\left(\int_{\lambda_{k}}^rg(s,x)ds\right) d E_L(\lambda_{k},r]f(x)dx\right|\nonumber\\
&\leq&\lim\sum_{j=1}^J \int_K \left(\int_{\lambda_{k}}^{\xi_j}|g(s,x)|ds\right) \left|E_L(\lambda_{k,j-1},\lambda_{k,j}]f(x)\right|dx\nonumber\\
&\leq&\lim\sum_{j=1}^J \|\int_{\lambda_{k}}^{\lambda_{k+1}}|g(s,x)|ds\|_{L^2(K)} \|E_L(\lambda_{k,j-1},\lambda_{k,j}]f\|_{L^2(K)}\nonumber\\
&=& \lim\sum_{j=1}^J\|E_L(\lambda_{k,j-1},\lambda_{k,j}]f\|_{L^2(K)}\nonumber\\
&\leq& C_K \lim\sum_{j=1}^J\lambda_{k,j}^{a-\frac{1}{2}} (\lambda_{k,j}-\lambda_{k,j-1})^{\frac{1}{2}}
 \|E_L(\lambda_{k,j-1},\lambda_{k,j}]f\|_2\nonumber\\
&\leq& C_K \lim\left(\sum_{j=1}^J\lambda_{k,j}^{2a-1} (\lambda_{k,j}-\lambda_{k,j-1}) \right)^{1/2}
\left(\sum_{j=1}^J\|E_L(\lambda_{k,j-1},\lambda_{k,j}]f\|^2_2\right)^{1/2}\nonumber\\
&\leq& C_K \left(\lambda_{k+1}^{2a-1} (\lambda_{k+1}-\lambda_{k}) \right)^{1/2} \|E_L(\lambda_{k},\lambda_{k+1}]f(x)\|_2
\end{eqnarray}
From \eqref{e4.4} and \eqref{e4.3}, we have that
\begin{align*}
&\left|\int_K\int_{[\lambda_{k},\lambda_{k+1}]} E_L(\lambda_{k},r]f(x)g(r,x)drdx\right|\\
&=\left|\int_K E_L(\lambda_{k},\lambda_{k+1}]f(x)
\left(\int_{\lambda_{k}}^{\lambda_{k+1}}g(s,x)ds\right)dx-\int_K\int_{\lambda_{k}}^{\lambda_{k+1}}
\left(\int_{\lambda_{k}}^rg(s,x)ds\right) d E_L(\lambda_{k},r]f(x)dx\right|\\
&\leq \|E_L(\lambda_{k},\lambda_{k+1}]f\|_2 +C_K \left(\lambda_{k+1}^{2a-1} (\lambda_{k+1}-\lambda_{k})
\right)^{1/2} \|E_L(\lambda_{k},\lambda_{k+1}]f\|_2.
\end{align*}
Recall that $\lambda_k=k^{1/(2a)}$ and it implies that $\lambda_{k+1}^{2a-1} (\lambda_{k+1}-\lambda_{k})\leq C$ with $C$ independent of $k$. Then
by \eqref{e4.2}, we see  that
\begin{align*}\label{forII}
  {\rm LHS\ of\ \eqref{e4.2} }
 &\leq C_{K} \sum_{k\in \NN} \left(\|E_L(\lambda_{k},\lambda_{k+1}]f\|^2_2 +C_K \lambda_{k+1}^{2a-1} (\lambda_{k+1}-\lambda_{k})
 \|E_L(\lambda_{k},\lambda_{k+1}]f\|^2_2\right)\nonumber\\
&\leq C_{K} \sum_{k\in \NN} \|E_L(\lambda_{k},\lambda_{k+1}]f\|^2_2 \nonumber\\
 &\leq C_{K}  \|f\|^2_2
 \leq C_{K} \|\log(2+L)f\|^2_2,
\end{align*}
which, together with \eqref{e4.1}, completes the proof of \eqref{e1.5} in  Theorem~\ref{th1.1}.
  \hfill{}$\Box$

\medskip

\section{Applications}
 \setcounter{equation}{0}

\subsection{Dirichlet operators on smooth bounded domains.} Let $\Omega$ be a connected bounded open subset of $\RN$ with $C^\infty$ boundary and $L=P(x,D)$
be a second order differential operator of the form
$$
P(x,D)=-\sum \frac{\partial}{\partial x_j}g^{jk}(x)\frac{\partial}{\partial x_k},
$$
where $(g^{jk})\in C^\infty(\Omega)$ is real and positive definite in $\bar \Omega$. Define the operator with Dirichlet boundary
 conditions. In \cite[Section 17.5]{Hor}, it is proved that the number $N(\lambda)$ of eigenvalues $\leq \lambda$ of $P(x,D)$ satisfies
$$
N(\lambda)=O(\lambda^{n/2}),
$$
and so condition \eqref{e1.6} holds.
From Theorem~\ref{th1.2},  we have the following proposition.
\begin{prop}
Let $\Omega$ be a connected bounded open subset of $\RN$ with $C^\infty$ boundary and $P(x,D)$ be a second order differential operator as above.
Assume that $\log(2+P)f\in L^2(\Omega)$. Then
$$
\lim_{R\to \infty} S_R(P) f(x)=f(x)
$$
for almost every $x\in \Omega$.
\end{prop}

\subsection{Schr\"odinger operators with growth potentials.} Assume that the potential $V: \RN\to \RR$ is smooth and satisfies the growth conditions:
\begin{eqnarray}\label{e5.1}
|\partial^\alpha V(x)|\leq C_\alpha (1+|x|)^k \mbox{\,for each multiindex \,}\alpha
\end{eqnarray}
and
\begin{eqnarray}\label{e5.2}
V(x)\geq c(1+|x|)^k \mbox{\,for\,} |x|\geq R,
\end{eqnarray}
where $k,c,C_\alpha,R>0$ are appropriate constants.

We consider the Schr\"odinger operator $-\Delta+V(x)$, where the potential $V$ satisfies the above conditions
\eqref{e5.1} and \eqref{e5.2}. An example is the Hermite operator $L=-\Delta +|x|^2.$
Then we have the following Weyl Law result:
$$
N(\lambda)\leq C\big|\{(\xi,x):|\xi|^2+V(x)\leq \lambda\}\big|,
$$
where $|\cdot|$ denotes the measure of the set in $\RR^{2n}$. See for example \cite[Section 6.4]{Zwo}.
Thus we have
$N(\lambda)\leq C\lambda^{n/2+n/k}$ , and   condition \eqref{e1.6} holds.
 From Theorem~\ref{th1.2}, we have the following proposition.
\begin{prop}
Let $L$ be the Schr\"odinger operator $-\Delta+V(x)$ where $V(x)$ satisfies the above growth condition.
Assume that $\log(2+L)f\in L^2({\mathbb R}^n)$. Then
$$
\lim_{R\to \infty} S_R(L) f(x)=f(x)
$$
for almost every $x\in{\mathbb R}^n$.
\end{prop}

\subsection{Schr\"odinger operators with   inverse-square potential.}
Now we consider the inverse square potentials, that is $V(x) = \frac{c}{| x |^2}$.  Fix $n \geq 3$ and assume that
$  -{(n-2)^2/4}< c $.     Define by quadratic form method
$L  = -\Delta + V$ on $L^2(\RR^n, dx)$.
The classical Hardy  inequality
\begin{equation}\label{hardy1}
- \Delta\geq  \frac{(n-2)^2}{4}|x|^{-2},
\end{equation}
shows that  for all $c > -{(n-2)^2/4}$,  the self-adjoint operator $L$  is   non-negative.
Set  $p_c^{\ast}=n/\sigma$, $\sigma= \max\{ (n-2)/2-\sqrt{(n-2)^2/4+c}, 0\}$. If $c \ge 0$ then the semigroup $\exp(-tL)$
is pointwise bounded by the Gaussian semigroup and hence act on all $L^p$ spaces with $1 \le p \le \infty$.  If $ c < 0$, then $\exp(-tL)$
acts as a uniformly bounded semigroup on  $L^p(\RR^n)$ for
$ p \in ((p_c^{\ast})', p_c^{\ast})$ and the range $((p_c^{\ast})', p_c^{\ast})$ is optimal (see for example  \cite{LSV}).

 It follows from  \cite[Theorem III.5]{COSY}  that $L$ satisfies the Plancherel-type estimate \eqref{e1.3}. From Lemma~\ref{le2.2}
  and Theorem~\ref{th1.1}, we obtain

\begin{prop}
Suppose that $n \geq 3$  and $ -{(n-2)^2/4} < c $. Let $L=-\Delta+c|x|^{-2}$ be defined as above.  Assume that $\log(2+L)f\in L^2({\mathbb R}^n)$. Then
$$
\lim_{R\to \infty} S_R(L) f(x)=f(x)
$$
for almost every $x\in {\mathbb R}^n$.
\end{prop}

\medskip

\noindent
\subsection{Scattering operators.}
Assume now that $n= 3$ and   $V$ is a real-valued measurable function such that
\begin{eqnarray}\label{eq10.1}
\int_{{\mathbb R}^6} \frac{|V(x)|\, |V(y)|}{|x-y|^2}dx dy < (4\pi)^2 \quad \ \ \mbox{and} \quad \ \
\sup_{x\in{\mathbb R}^3}
\int_{{\mathbb R}^3} \frac{ |V(y)| }{|x-y|}dy< 4\pi.
\end{eqnarray}
Suppose that $L =-\Delta  + V$ on ${\mathbb R}^3$
 with a real-valued    $V$ which satisfies   \eqref{eq10.1}.

From  \cite[Proposition III.6]{COSY}, we know that $L$ satisfies the Plancherel-type estimate \eqref{e1.2} and \eqref{e1.3}. From Lemma~\ref{le2.2}
  and Theorem~\ref{th1.1},  we have the following  proposition.

\begin{prop}
Let $L=-\Delta+V(x)$ be defined on $\RR^3$ as above.  Assume that $\log(2+L)f\in L^2({\mathbb R}^3)$. Then
$$
\lim_{R\to \infty} S_R(L) f(x)=f(x)
$$
for almost every $x\in {\mathbb R}^3$.
\end{prop}

\medskip

 \noindent
{\bf Acknowledgments.}    P. Chen was supported by NNSF of China 12171489.
 X.T. Duong was supported by  the Australian Research Council (ARC) through the research
grant DP190100970.
 L. Yan was supported by the NNSF of China
 11871480  and by the Australian Research Council (ARC) through the research
grant DP190100970.

\end{document}